\documentclass[fleqn,final,5p,times,twocolumn]{elsarticle}

\journal{Systems \& Control Letters}

\usepackage{graphics}
\usepackage{subfigure}
\usepackage{balance}
\usepackage{amsmath,amssymb,amsfonts}
\allowdisplaybreaks
\usepackage{hyperref}
\usepackage{cleveref}
\hypersetup{pdfauthor=author}
\usepackage{color}
\usepackage{float}
\usepackage{mathrsfs}  
\usepackage{amsmath}
\usepackage{amsthm}
\usepackage{subfigure}
\usepackage{graphicx}
\usepackage{epstopdf}
\usepackage[all]{xy}
\usepackage{latexsym,amscd}
\usepackage{url}
\usepackage{ulem}
\usepackage{cancel}
\usepackage{dsfont}

\newtheorem{theorem}{Theorem}[section]

\newtheorem{proposition}[theorem]{Proposition}
 
\theoremstyle{definition}  
\newtheorem{definition}[theorem]{Definition}

\newtheorem{remark}[theorem]{Remark}

\newtheorem{Lemma}[theorem]{Lemma}

\newtheorem{Example}[theorem]{Example}
\newtheorem{Remark}[theorem]{Remark}

\newcommand{\F}{{\mathcal{F}}}
\newcommand{\p}{\mathrm{p}}
\newcommand{\on}{\operatorname}   
\newcommand{\R}{{\mathbb{R}}}

\begin{document}
\begin{frontmatter}
\title{Realization Theory Of Recurrent Neural ODEs Using Polynomial System Embeddings}
\author[1]{Martin Gonzalez\corref{author}}
\address[1]{Institut de Recherche Technologique SystemX, Palaiseau, France.}
\cortext[author] {Corresponding author. \textit{E-mail address:} martin.gonzalez@irt-systemx.fr}
\author[3]{Thibault Defourneau \fnref{f3}}
\address[3]{Trinov, 75001 Paris, France. }
\fntext[f3]{The work presented in this paper was carried out while 
T. Defourneau was employed by Centre de Recherche en Informatique, Signal et Automatique de Lille.}
\author[1]{Hatem Hajri}
\author[2]{Mihaly Petreczky}
\address[2]{Centre de Recherche en Informatique, Signal et Automatique de Lille, UMR CNRS 9189,  France.}
\begin{abstract}
In this paper we show that neural ODE analogs of recurrent (ODE-RNN) and
Long Short-Term Memory (ODE-LSTM) networks can be algorithmically embedded into the class of 
polynomial systems. This embedding preserves input-output
behavior and can suitably be extended to other neural DE architectures. We then use realization theory of
polynomial systems to provide necessary conditions for an input-output map to be realizable by an 
ODE-LSTM and sufficient conditions for minimality of such systems.
These results represent the first steps towards realization theory of recurrent neural ODE 
architectures, which is is expected be useful for model reduction and learning 
algorithm analysis of recurrent neural ODEs. 
\end{abstract}

\end{frontmatter}

\section*{Introduction}
Long Short-Term Memory networks (LSTMs) represent a generalization of 
recurrent neural networks (RNNs) widely used in text analysis tasks such as grammar correction and 
next word prediction.  They were introduced \cite{LSTMorigine}, and have been studied intensively 
ever since, due to their effectiveness for learning long-term dependencies in comparison with RNN 
algorithms and its variants (see \cite{RNNsproblems1} and \cite{RNNsproblems2} for details). 

More recently, the introduction of neural ODEs \cite{chen1} and other implicit network architectures \cite{deq,sdeq,ruba, Kidg1,Biau,sde,bsde,pde} 
has opened the door to new machine learning paradigms tightly related to dynamical system modeling 
techniques.  For instance, neural ODEs can be seen as  \emph{continuous-time} dynamical systems and 
as infinitesimally connected ResNets while the latter can be seen as Euler discretizations of 
neural ODEs.  As such, they exploit a rich available theory on both sides, offering memory efficiency 
whilst their recurrent analogs have the ability of handling irregular data and are suitable for 
tackling generative problems and time series (particularly in physics),  becoming relevant to both 
modern machine learning and traditional mathematical modeling. 

In this paper we make the first steps towards developing realization theory of recurrent neural ODE architectures.  We focus on the present article on neural ODE analogs of RNNs and LSTMs and we aim at characterizing those input-output maps which can be represented by these systems and understanding 
the minimal size of such systems sufficient to be able to represent a given input-output map. 

The motivation for studying realization theory for neural ODEs is that learning algorithms 
for such systems from data correspond to system identification 
algorithms. Realization theory is central in system identification as it can be viewed as an 
attempt to solve a system identification problem through idealized qualitative analysis, where 
there is infinite data and no modelling error.
For linear systems, realization theory \cite{kailath,LindquistBook} allowed to address 
identifiability, canonical forms and gave rise to subspace identification algorithms. 

In order to make the discussion more precise, 
let us determine an ambient class of dynamical systems containing ODE-RNNs,  
ODE-LSTMs\footnote{These will be specified by a general class of two-scale ODE-based RNNs where multiple continuous-time memory architectures arise.} and polynomial systems as some of its subclasses, and 
having all desired properties needed for our study. 
For a subset $Z\subset \R^k$, $k \geq 1$, we denote: 
\begin{itemize}
\item $\mathcal{Z}$ 
(resp. $\mathcal{Z}_{\on{pc}}, \mathcal{Z}_{\on{ac}}$ ) 
the set of continuous (resp. piecewise-continuous, resp. absolutely continuous  
) functions from $[0; + \infty[$ to $Z$. Denote by $\mathcal{Z}_{\on{pco}}$ the set of all 
piecewise-constant 
functions from $[0,+\infty[$ to $Z$ which are constant starting from a certain point, i.e., $h \in \mathcal{Z}_{\on{pco}}$, 
if $h$ is piecewise-constant and there exists $T_h \ge 0$ such that the restriction of $h$ 
to $[T_h,+\infty[$ is constant. 
\item $\mathcal{Z}^f$  the set of functions from $[0, T]$ to $Z$, for some 
$T>0$.  
We let the reader combine this with the above notations.
\end{itemize}
The ambient class $\mathcal{F}$ of dynamical systems we will 
consider in this paper is described by differential equations of the form  
\begin{equation}\label{DNs}
\left\lbrace \begin{array}{lll} 
\dot{x}(t) =f\big(x(t),u(t) \big) \\
y(t) = g\big( x(t)\big) \end{array} \right.  \quad t\geq 0,
\end{equation}
with initial condition $x(0) = x_0 \in \R^n$, where
\begin{itemize}
\item $(x,u,y)\in\mathcal{X}_{ac} \times\mathcal{U}_{pco} \times \mathcal{S}_{pc}$
\footnote{Allowing $u\in \mathcal{U}_{pc}$ represents a relatively small technical 
difficulty that we avoid here for clarity.}, where 
$X=\R^n,S=\R^{\mathrm{p}}$ denote respectively the state and output 
spaces 
and $U=\{\alpha_1,\ldots,\alpha_K\} \subset \R^m $ is a finite 
input space with cardinality $K$;
\item $f: X \times U \rightarrow X$ is analytic on its first
argument and $g:X \rightarrow S$ is analytic.
\end{itemize}

We will identify these systems with tuples of the form $\Sigma = (f,g, x_0)$. 
The triple $(m,n,\mathrm{p})$ will be called the format of $\Sigma$.


On the one hand,  polynomial systems are a subclass of such systems and many 
methods of computational algebra can be used to determine qualitative 
properties of such systems such as  observability, reachability and minimality. 
On the other hand,  one can think of ODE-RNNs and ODE-LSTMs as subclasses of $\mathcal{F}$ 
which can be parameterized according to some class of learning weight functions $\theta(t)$, 
which will be assumed to be constant for simplicity.  As such, under mild assumptions, 
we can associate polynomial 
systems to large classes of ODE-RNNs
and ODE-LSTMs and,  by doing so,  infer such 
qualitative properties on these classes.  More specifically,
\begin{itemize}
\item We show that an i-o map can be realized by an ODE-RNN or an ODE-LSTM, only if 
it can be realized by a polynomial system, i.e. a non-linear system defined by 
vector fields and readout maps which are polynomials. We present an explicit 
algorithmic construction of such a polynomial system.
\item We infer sufficient conditions for minimality/observability/reachability/accessibility of 
ODE-RNNs and ODE-LSTMs from the properties of their associated polynomial systems
 \cite{JanaSIAM,Jana,Bartoszewicz,SontagWang}.
\item We present a necessary condition for existence of a realization by ODE-RNNs and ODE-LSTMs, 
using results from realization theory of polynomial systems. This necessary condition is a 
generalization of the well-known rank condition for Hankel matrices of linear systems.
\end{itemize}

Note that elements in $\F$ could be viewed as analytic systems for which there is an 
existing realization theory \cite{J,Isidori,NLCO}. 
However,  as analytical functions do not have a finite representation, this approach is not 
computationally effective: there are no algorithms for checking minimality, deciding 
equivalence of two systems neither transforming a system to a minimal one. Nevertheless, 
seen as polynomial systems,  computer algebra tools can be used to address these issues
\cite{JanaCDC2016}. In addition, since polynomial systems have much more algebraic 
structures than analytic systems and the conditions for  minimality/observability/reachability 
studied here are less restrictive than those which can be obtained by the analytic 
approach.

\textbf{Related work: }  To the best of our knowledge, the results of the paper are new.  
Observability, controllability and minimality of ODE-RNNs were investigated in \cite{ASo,QSc,ASm},
but no results on existence of a realization were provided, and the results of \cite{ASo,QSc,ASm}
used certain assumptions on the weights of the ODE-RNNs.
In contrast to \cite{ASo,QSc,ASm}, in this paper we consider ODE-LSTMs and we address the issue of existence of a
realization by ODE-LSTM. Moreover, the technique used in this paper is completely different  from
that of \cite{ASo,QSc,ASm}. 
Rational embeddings and elements of realization theory 
of a subclass of ODE-RNNs were considered in \cite{TMArxive}.  In comparison to
\cite{TMArxive}, the main novelty is that in this paper we consider both ODE-RNNs and ODE-LSTMs  and 
that detailed proofs and examples are provided. That is, the current paper extends the results of \cite{TMArxive}. 

\section{Preliminaries}

We denote $\R[X_1, \dots ,X_n]$ the algebra of real polynomials in $n$ 
variables and
denote $\R(X_1,\ldots,X_n)$ its quotient field, whose elements are rational 
functions in $n$ variables. 
If $R$ is an integral domain over $\R$ then the transcendence degree 
$\mathrm{trdeg} R$ of $R$ over $\R$ is defined as the transcendence 
degree over $\R$ of the field $F$ of fractions of $R$ and it equals the 
greatest number of algebraically independent elements of $F$ over $\R$.
Let $n,m$ be two integers and let $\R[X_1,\ldots,X_n;\R^m]$ be the set of 
tuples $(P_1,\ldots,P_m)$ whose $m$ components are polynomials in $n$ 
variables.  

Denote $\mathcal{C}^{\omega}(\R)$ the algebra of real analytic functions 
over $\R$ and  denote $\sigma^{(i)}$ the $i$-th derivative of 
$\sigma\in\mathcal{C}^{\omega}(\R)$. We denote 
$\mathcal{C}_1^{\omega}(\R)$ the subset of $\mathcal{C}^{\omega}(\R)$ of those
analytic functions satisfying 
\begin{equation}\label{assum1}
 \sigma^{(1)} = P(\sigma) \, , 
\end{equation}
for some $P \in \R[X]$. In this paper we will consider only activation functions of this sort.
In particular, the hyperbolic tangent and the logistic functions
\begin{equation*}
\forall x \in \R \, , \hspace*{2mm} th(x) = \frac{e^x - e^{-x}}{e^x + e^{-x}} \, , 
\hspace*{2mm} S(x) = \frac{1}{1 + e^{-x}} \, 
\end{equation*}
are both elements of $ \mathcal{C}_1^{\omega}(\R)$ as they are the unique solutions of the 
differential equations
$$
th'(x)=1-th^2(x) \qquad S'(x)=S(x)-S^2(x)
$$
with initial conditions $th(0)=0$ and $S(0)= \frac{1}{2}$.

For a map $\sigma:\R \rightarrow \R$ denote 
$\overrightarrow{\sigma}: \R^n \rightarrow \R^n$ the map defined by 
\begin{equation}\label{diags}
\overrightarrow{\sigma}: (x_1, \ldots, x_n)^T \mapsto (\sigma(x_1), \ldots, 
\sigma(x_n))^T 
\end{equation}

Let $\odot$ be the Hadamard product $(A \odot B)_{ij} := (A)_{ij} (B)_{ij}$ 
where $A,B$ matrices of same dimension.  In particular,  for $P,Q,R \in \R^n$,  
the expression $P \odot Q= R$ is equivalent to $P_i Q_i= R_i$,  for
$i\in[n]:=\{1,\ldots,n\}$.

\subsection{Polynomial systems}

Define the subclass $\mathcal{F}_0$ of $\mathcal{F}$ consisting of those 
systems described by 
\begin{equation} \label{rational_system_R(Sigma)}
 \left\lbrace\begin{array}{ll}
\dot{x}(t) = P_{u(t)}(x(t)) \, \\[2mm]
y(t)  = g(x(t)) \, ,
\end{array} \right.\quad t\geq 0,
\end{equation}
with initial condition $x(0) = x_0\in \R^n$, where 
$P_{u(t)}\in \R[X_1,\ldots,X_n;\R^{n}]$ for $u(t) \in U$ and 
$g\in  \R[X_1,\ldots,X_n;\R^{\p}]$.
These are polynomial systems and will be identified with tuples 
$\mathscr{P} = (\{P_{u(t)}\}_{u(t) \in U},g, x_0)$ and have at most one solution $(x,u,y)\in\mathcal{X}_{\on{ac}}\times\mathcal{U}_{\on{pco}}
\times \mathcal{S}_{\on{pc}}$ given an initial state $x_0$.

\begin{definition}\label{poly-systems}
We say that a polynomial system $\mathscr{P}\in \mathcal{F}_0$ is a  
polynomial embedding of a system $\Sigma\in \mathcal{F}$ if for any solution 
$(x,u,y)\in\mathcal{X}_{\on{ac}}\times\mathcal{U}_{\on{pco}}\times 
\mathcal{S}_{\on{pc}}$ of $\Sigma$, there is a continuous injection $F$ 
such that $(F(x),u,y)\in F(\mathcal{X}_{\on{ac}})\times\mathcal{U}_{\on{pco}}
\times \mathcal{S}_{\on{pc}}$ is a solution of $\mathscr{P}$. We will denote 
such systems by $\mathscr{P}(\Sigma)$.
\end{definition}

\subsection{The algebra of input-output maps}

In this section we introduce causal analytic i-o maps and verify that under some assumptions their observation algebras are well-defined.  For the latter, we make use of some technical definitions (see \cite[Definitions 4.2, 4.3]{Jana}), allowing us to define derivations and show that the ring of input-output maps forms an integral domain structure.

To this end, we remark that any element $u$ of $\mathcal{U}_{\on{pco}}$ is completely determined by 
determined by tuples $t_1,\ldots,t_l\in[0,T],  \alpha_1,\ldots,\alpha_l\in U$ for 
some $l\geq1$ such that:
$$
u_{t_1,\ldots,t_l}^{\alpha_1,\ldots,\alpha_l}(t)=\left\{\begin{array}{rl}
 \alpha_i & \mbox{ if } \; t \in [T_{i-1}, T_i[, \hspace*{1mm}  i\in[l] \\
 \alpha_l & \mbox{ if } \; t \ge T_l
 \end{array}\right.  
$$
for some interval decomposition 
$$
T_0=0 \, , \hspace*{1mm} T_i=\sum_{j=1}^{i} t_j, 
\hspace*{1mm} i\in [l] \quad T=T_l.
$$
Now, for $p:\mathcal{U}_{\on{pco}} \to \mathcal{S}_{\on{pc}}$, each component of 
$p(u):=(p_1(u),\ldots,p_{\on{p}}(u))$ is of the form 
$p_k(u_{t_1,\ldots,t_l}^{\alpha_1,\ldots,\alpha_l})$. 
This remark will be used to define the class of causal and analytic input-output maps. 
In turn, any input-output map realized by a system from $\mathcal{F}$ 
belongs to this latter class.
\begin{definition} 
A map $p: \mathcal{U}_{pco} \rightarrow \mathcal{S}_{pc}$ is 
\begin{itemize}
\item \emph{causal} if, $\forall u, v \in \mathcal{U}_{pco}, t \geqslant 0, $ we have:
$$ \left(u(s) = v(s), \forall s \in [0,t] \right) \Rightarrow  \left(p(u)(s) = p(v)(s), 
\forall s \in [0,t]\right).$$
\item \emph{analytic} if $\forall k \in [\p], \alpha_1,\ldots,
\alpha_l \in \mathcal{U},l > 0$, 
the following function is analytic 
\begin{eqnarray*}
\phi_{p,k,\alpha_1,\ldots,\alpha_l}: ([0,+\infty[)^l & \rightarrow &  \R \\
(t_1, \ldots, t_l) & \mapsto & 
p_k(u_{t_1,\ldots,t_l}^{\alpha_1,\ldots,\alpha_l})(T_l).
\end{eqnarray*}
\end{itemize}
Set $S^0=\R$.  Denote $\mathcal{A}(\mathcal{U}_{pco})$ the set of causal 
analytic maps 
$p:\mathcal{U}_{pco} \rightarrow \mathcal{S}^0_{pc}$. It is naturally a $\R$-
algebra. 
\end{definition}

For our purposes, we need to define a derivation operation  on $\mathcal{A}(\mathcal{U}_{pco})$. 
To this end, we will use the following observation: for each $\alpha\in U, t>0$ and 
$u\in\mathcal{U}_{\on{pco}}$, 
we can construct an element $u_{t,\alpha}\in \mathcal{U}_{\on{pco}}$ 
by setting
$$
u_{t,\alpha}(\tau)=\left\{\begin{array}{rl} 
                     u(\tau) & \tau \in [0,t[ \\
                     \alpha  & \tau \geq t. \\
     \end{array}\right.     
$$
Then for all $\alpha\in U$, we define the map
$$
D_{\alpha}: \mathcal{A}(\mathcal{U}_{pco}) \rightarrow  \mathcal{A}
(\mathcal{U}_{pco}) 
\quad \varphi  \mapsto D_{\alpha}(\varphi)
$$
given, for all $u\in \mathcal{U}_{pco}, t \geq 0$ by,
$$
\big( D_{\alpha} \varphi(u) \big)(t) = \frac{d}{ds} \Big( \varphi \big( u_{t,\alpha})
(t+s)\Big)_{| \, s = 0}.
$$
The map $D_{\alpha}$ is a well-defined derivation.

Next, we will argue that  $\mathcal{A}(\mathcal{U}_{pco})$ is an integral domain. 
To this end notice that the set $\mathcal{U}^f_{\on{pc}}$ of piecewise constant 
functions over finite intervals is closed by interval truncation, concatenation 
and piecewise time dilatation. Hence, $\mathcal{U}^f_{\on{pc}}$ is a set of admissible 
inputs in the sense of \cite[Definition 4.1]{Jana}. 
Consequently, we can use the definition of analytic functions in the sense of \cite[Definition 4.3]{Jana}.
Let us denote by $\mathcal{A}(\mathcal{U}^f_{\on{pc}},\R)$ the set of analytic 
functions $\mathcal{U}^f_{\on{pc}} \to \R$ in the sense of \cite[Definition 4.3]{Jana}.
From \cite[Theorem 4.4]{Jana} it follows that the ring 
$\mathcal{A}(\mathcal{U}^f_{\on{pc}},\R)$ is an integral domain. 
Below we will present an $\mathbb{R}$-algebra isomorphism between $\mathcal{A}(\mathcal{U}_{pco})$ and $\mathcal{A}(\mathcal{U}^f_{\on{pc}},\R)$. The existence of such an isomorphism then implies
that  $\mathcal{A}(\mathcal{U}_{pco})$ is also an integral domain.
In order to define this isomorphism, we observe that for each 
$v:[0,T_v]\to U$ in $\mathcal{U}^f_{\on{pc}}$ we can construct a unique 
element $u_{v}\in \mathcal{U}_{\on{pco}}$ by setting
$$
u_{v}(t)=\left\{\begin{array}{rl} 
                     v(t) & t \in [0,T_v[ \\
                     v(T_v)  & t \geq T_v. \\
     \end{array}\right.     
$$
Let us define the map 
\begin{eqnarray*}
 \ell:\mathcal{A}(\mathcal{U}_{pco}) & \longrightarrow &  \mathcal{A}
(\mathcal{U}^f_{\on{pc}},\R) \\
\Big(\varphi: \mathcal{U}_{pco} \to \mathcal{S}^0_{pc} \Big) & \longmapsto & 
\Big( \tilde{\varphi}:\mathcal{U}^f_{\on{pc}} \to \R \Big)
\end{eqnarray*}
given, for each function $v:[0,T_v]\to U$ in $\mathcal{U}^f_{\on{pc}}$,  by 
$\widetilde{\varphi}(v)=\varphi(u_v)(T_v)$. 

It then follows $\ell$ is an $\R$-algebra isomorphism and hennce  $\mathcal{A}(\mathcal{U}_{pco})$ is an
integral domain.

The above discussion allow us to introduce the following.

\begin{definition} \label{observation_algebra_field}
Let $p: \mathcal{U}_{pco} \rightarrow \mathcal{S}_{pc}$ be 
analytic and causal.  The \emph{observation algebra} $\mathcal{A}_{obs}(p)$ 
of $p$ is the smallest sub-algebra  
$\mathcal{A} \big( \mathcal{U}_{pco} \big)$ 
containing each $p_k \in \mathcal{A}_{obs}(p)$ and closed under 
$D_{\alpha}$, for all $\alpha\in U$. 
The field of fractions $\mathcal{Q}_{obs}(p)$ of $\mathcal{A}_{obs}(p)$ 
will be called observation field of $p$ and we denote 
$trdeg\mathcal{A}_{obs}(p)$ the transcendence degree 
of $\mathcal{A}_{obs}(p)$ over $\R$.
\end{definition}

\begin{definition}
Let $\Sigma \in \F$ be a system with initial state $x_0$.  
It is called a (piecewise constant) realisation of a map 
$p: \mathcal{U}_{\on{pco}} \rightarrow \mathcal{S}_{pc}$
if for all $u\in\mathcal{U}_{\on{pco}}$ the unique solution $(x,u,y)$ of 
$\Sigma$ such that $x(0)=x_0$
satisfies $p(u) \equiv y$.
\end{definition}

\begin{remark}
If a system $\Sigma$ realizes an i-o map 
$p: \mathcal{U}_{pco} \rightarrow \mathcal{S}_{pc}$, then the polynomial 
embedding $\mathscr{P}(\Sigma)$, when it exists, also realizes $p$.
\end{remark}

\subsection{Minimality, reachability and observability of polynomial systems} 

Let $\Sigma\in \F$ be a system of format $(m,n,\p)$ as in the above subsection.
The \emph{dimension} $dim(\Sigma)$ of $\Sigma$ is the dimension of its 
state-space.

A polynomial system $\mathscr{P}$ realizing an i-o map $p$ is \textit{minimal} if 
there is no polynomial system $\mathscr{P}^{'}$ realizing $p$ such that 
$\dim(\mathscr{P}^{'}) < \dim(\mathscr{P})$. 
Define the set of reachable states of a polynomial system $\mathscr{P}$ as:
\begin{equation*} 
 \mathrm{R}_{\mathscr{P}}(\upsilon_0) = \{ \upsilon(t) \; | \; t \geqslant 0, 
 (\upsilon,u,y) 
 \mbox{ is a solution of } \mathscr{P}, \upsilon(0)=\upsilon_0\} \, 
\end{equation*}
and recall from  \cite[Definition 4]{Bartoszewicz1} that its \emph{observation 
algebra} $\mathcal{A}_{obs}(\mathscr{P})$ is the smallest sub-algebra of the ring 
$R[X_1,\ldots,X_n]$ which contains $h_k$, $k\in[p]$
and which is closed under taking the formal Lie derivatives with respect to the 
formal vector fields
$f_{\alpha}=\sum_{i=1}^{n} P_{i,\alpha} \frac{\partial}{\partial X_i}$.  
Its fraction field $\mathcal{Q}_{obs}(\mathscr{P})$ will be called its observation field.
Finally, $\mathscr{P}$ is minimal if  $\dim(\mathscr{P})= \mathrm{trdeg} 
\mathcal{A}_{obs}(p)$ (see \cite[Lemma 1, Theorem 4]{JanaVS} for details).  
Notice that the other implication is true for rational systems, but not for polynomial ones.
A polynomial system $\mathscr{P}$ is
\begin{itemize}
\item \emph{algebraically reachable}, if there is no non-trivial polynomial which is 
zero on
$\mathrm{R}_{\mathscr{P}}(\upsilon_0)$;
\item \textit{accessible}, if $\mathrm{R}_{\mathscr{P}}(\upsilon_0)$ contains an 
open subset of $\R^n$;
\item  \emph{algebraically observable}, if 
$\mathcal{A}_{obs}(\mathscr{P}) = R[X_1,\ldots,X_n]$ (see \cite{Jana});
\item \emph{semi-algebraically observable} if $\mathrm{trdeg}(\mathcal{A}_{obs}
(\mathscr{P}))=n$ (see \cite{NashSIAM});
\item \emph{observable}, if for every two distinct initial states 
$\upsilon_0,\upsilon_0^{'}$ there exists solutions $(\upsilon,u,y)$ and 
$(\upsilon^{'},u,y^{'})$ of $\mathscr{P}$ such that 
$\upsilon(0)=\upsilon_0$, $\upsilon^{'}(0)=\upsilon_0^{'}$, and $y \ne y^{'}$. 
\end{itemize}
Accessibility implies algebraic reachability and algebraic observability implies 
semi-algebraic observability, and semi-algebraic observability implies
observability. 
A polynomial system $\mathscr{P}$ is minimal if it is algebraically 
reachable and algebraically observable (see \cite[Theorem 4]{JanaVS} and 
\cite{Bartoszewicz} for details).  Notice that the other implication is true for 
rational systems but not for polynomial ones.
Algebraic, rational and semi-algebraic observability and algebraic reachability of 
polynomial systems can be checked using methods of computational algebra 
\cite{JanaCDC2016}. 

\medskip

Define the reachable set of a system $\Sigma$ in $\F$ of format $(m,n,\p)$ by
\begin{equation*}
 \mathrm{R}_{\Sigma}(s_0) = \{s(t)  \; | \; t \geqslant 0 \, , (s,u,y) 
 \mbox{ is a solution of } \Sigma, s(0)=s_0  \}. 
\end{equation*}
We will say that $\Sigma$ is 
\begin{itemize}
\item \emph{accessible}, if $\mathrm{R}_{\Sigma}(s_0)$ contains an open subset 
of $\R^{n}$;
\item \emph{algebraically reachable} if there is no non-trivial polynomial which is 
zero on $R_{\Sigma}(s_0)$;
\item \emph{span-reachable}, if the linear span of the elements 
$\mathrm{R}_{\Sigma}(s_0)$ is $\R^{n}$;
\item reachable if there exist no linear function which is zero on 
$\mathrm{R}_{\Sigma}(s_0)$;
\item \emph{weakly observable} if for every initial state $\hat{s} \in \R^n$ there is 
an open subset $V$ of $\R^n$ such that $\hat{s} \in V$ and for every 
$\hat{s} \ne \overline{s} \in V$, there exist solution $(s,u,y)$ and $(s',u,y')$ of 
$\Sigma$, with $s(0)=\hat{s}$ and $s'(0)=\overline{s}$, such that $y \neq y'$;
\item \emph{observable} if for every initial state  $\hat{s} \in \R^n$, $V = \R^n$ in 
the latter definition.
\end{itemize}
Accessibility implies algebraic reachability which in turn implies span-reachability. 
Observability implies weak observability.
Finally, if the system $\Sigma$ realizes an i-o map $p$,  is accessible and weakly 
observable, then it is minimal dimensional among all the systems from $\mathcal{F}$ 
realizing $p$ (see \cite[Theorem 1.12]{J}).

\section{Realization theory of dynamical neural networks}

\subsection{Recurrent neural nets and LSTM embeddings}

Let us denote by $\mathcal{F}_1 \subset \mathcal{F}$ the class of systems 
described by differential equations
\begin{equation} \label{RNN_equation}
\Sigma: \left\lbrace \begin{array}{ll} 
\dot{x}(t) = \overrightarrow{\sigma} \big( A x(t) + B u(t) \big)\\[2mm]
y(t) = C x(t) \end{array} \right.  \quad t\geq 0,
\end{equation}
with initial condition $x(0) = x_0\in \R^n$ and where
\begin{itemize}
\item $\sigma \in \mathcal{C}_1^{\omega}(\R)$ is  Lipschitz continuous,
\item $A \in \R^{n \times n}$, $B \in \R^{n \times m}$ and 
$C \in \R^{\mathrm{p} \times n}$ are matrices.
\end{itemize}

\begin{definition} \label{RNNs}
An element of $\mathcal{F}_1$ will be called a recurrent neural 
ODE (ODE-RNN). We will identify such systems with tuples  
$\Sigma = (A, B, C, \sigma, x_0)$ and the triple $(m,n,\p)$ will be its format.
\end{definition}


Let us denote by $\mathcal{F}_2$ the subclass of $\F$ of systems described by 
differential equations of the form
\begin{equation} \label{LSTM_equation}
\Sigma: \left\lbrace \begin{array}{llll} 
\dot{x}(t) = \mathfrak{U}^0 x(t) + g^2(t) \odot x(t) + g^3(t) \odot g^1(t) \\[2mm]
\dot{z}(t) =g^4(t)\\[2mm]
y(t) = C s(t) \end{array} \right.  \quad t\geq 0,
\end{equation}
with initial condition $s(0) = s_0=(x_0^T,z_0^T)^T \in \R^{2n}$, where 
\begin{itemize}
\item $g^i(t):=\overrightarrow{\sigma_i}(\mathfrak{U}^i h(t)+ 
\mathfrak{W}^i u(t)+b^i)$, for $i\in[4]$,
\item $u(t) \in \R^m$,  $s(t) = (x(t)^T, z(t)^T)^T \in \R^{2n}$, $y(t) \in \R^p$,
\item $h(t) = z(t) \odot \overrightarrow{\sigma_5}(x(t))$,
\item $\sigma=\{\sigma_1, \sigma_2, \sigma_3, \sigma_4, \sigma_5\}\subset 
C_1^{\omega}(\R)$ are all Lipschitz continuous,
\item $\mathfrak{U} = \{\mathfrak{U}^0,\mathfrak{U}^1, \mathfrak{U}^2, \mathfrak{U}^3, 
\mathfrak{U}^4\} \subset \R^{n \times n}$, 
\item $\mathfrak{W} = \{\mathfrak{W}^1, \mathfrak{W}^2, \mathfrak{W}^3, \mathfrak{W}^4\} 
\subset \R^{n \times m}$, and 
$C \in \R^{\mathrm{p} \times 2n}$,
\item  $\mathfrak{b} = \{b^1, b^2, b^3, b^4\} \subset \R^n$.
\end{itemize}

\begin{definition} \label{LSTM}
An element of $\mathcal{F}_2$ will be called a long short-term neural ODE
(ODE-LSTM). We identify such systems with tuples  
$\Sigma = (\mathfrak{U}, \mathfrak{W}, \mathfrak{b}, C, \sigma, s_0)$, set the triple 
$(m,2n,\p)$ to be its format and define the (ordered) set $\sigma$ to be its activation.
\end{definition}

We will restrict our attention to solutions of systems $\Sigma$ in 
$\F_1$ or in $\F_2$ of the form $(x,u,y)\in\mathcal{X}_{\on{ac}}\times
\mathcal{U}_{\on{pco}}\times \mathcal{S}_{\on{pc}}$ and we recall 
that activations $\sigma$ are assumed to be in
 $C_1^{\omega}(\R)$ so that global existence and uniqueness of solutions is verified and
determined by $u$ and for some initial value (see \cite{TMArxive} for details). 
Notice that the Euler discretization of such systems correspond exactly to 
residual RNNs and residual LSTMs (see \ref{eqn:resnet} for details).

\begin{theorem}
\label{Theoremi:main}
All ODE-RNNs and ODE-LSTMs in $\F_1$ and $\F_2$ have polynomial embeddings.  
Moreover, if an i-o map $p$ 
has a realization by a ODE-LSTM, or a ODE-RNN then $p$ is causal, analytic and 
$\mathrm{trdeg}~  A_{obs}(p) < +\infty$.
\end{theorem}
The proof of this theorem will be done in  Subsection \ref{Proof1}.  

\medskip

We highlight the fact that our definitions of polynomial embeddings are completely 
explicit and can be easily implemented algorithmically. Theorem 
\ref{Theoremi:main} allows to infer qualitative properties on neural nets induced by 
properties of their polynomial embeddings as we will show in the next section.

\subsection{Qualitative properties of ODE-LSTMs} 
\label{sect:min}

A system $\Sigma$ in $\F_2$ with given activation
$\sigma$ realizing an i-o map $p$ is said to be \emph{$\sigma$-minimal} there 
exists no $\Sigma^{'}$ in $\F_2$ with activation $\sigma$,  
such that $\Sigma^{'}$ is a realization of $p$ and 
$\dim(\Sigma^{'}) < \dim(\Sigma)$.

\begin{Lemma}
\label{minimality_1}
Assume that an i-o map $p$ is realized by a system $\Sigma$ in $\F_2$  with 
given activation function $\sigma$.  If one of the conditions 
below holds, then $\Sigma$ is a $\sigma$-minimal realization of $p$: 
\begin{enumerate}
\item $\mathscr{P}(\Sigma)$ is a minimal realization of $p$,
\item $\mathrm{trdeg} \mathcal{A}_{obs}(p)=\on{dim}(\Sigma)$,
\item $\mathscr{P}(\Sigma)$ is semi-algebraically observable and algebraically 
reachable,
\item $\mathscr{P}(\Sigma)$ is algebraically observable and accessible.
\end{enumerate}
\end{Lemma}

\begin{proof}
The second point comes from the first and from \cite[Proposition 6]{JanaVS}. The 
rest of the proof is straightforward.
\end{proof}

\begin{proposition} \label{RNN_ReachObs}
Let  $\Sigma$ be a system in $\F_2$.
\begin{enumerate}
\item If $\mathscr{P}(\Sigma)$ is accessible, then $\Sigma$ is also accessible.  
\item
If $\mathscr{P}(\Sigma)$ is algebraically reachable, then $\Sigma$ is span-
reachable.  In particular, if $\mathscr{P}(\Sigma)$ is accessible, then  $\Sigma$ is 
span-reachable. 
\item
If $\mathscr{P}(\Sigma)$ is observable, then $\Sigma$ is observable. In particular, if 
$\mathscr{P}(\Sigma)$ is algebraically observable, then $\Sigma$ is observable. 
\item If $\mathscr{P}(\Sigma)$ is semi-algebraically observable, then $\Sigma$ is 
weakly observable. 
\end{enumerate}
\end{proposition}

The proof of the above proposition will be done in Subsection \ref{Proof2}. Notice that a 
similar result was established for $\F_1$ in \cite{TMArxive} with the help of an auxiliary 
polynomial embedding.

\medskip

Note that the variables involved in the polynomial embedding $\mathscr{P}
(\Sigma)$ can be naturally reordered, or simply reduced, following the expressions 
of the activation functions $\sigma$ of $\Sigma$. 

Accessibility and algebraic/semi-algebraic observability conditions for rational/
polynomial systems can be checked by using methods of computer algebra 
\cite{JanaCDC2016}. In contrast, for checking accessibility and (weak) 
observability of an ODE-LSTM the only systematic tools are the rank conditions 
\cite[Theorems 2.2,  2.5,  3.1,  3.5]{NLCO} or 
\cite[Corallaries 2.2.5,2.3.5]{Isidori}, which are not computational effective 
manner for analytic $\sigma$. 
Notice that minimality of $\mathscr{P}(\Sigma)$ is a much weaker condition than
accessibility and weak observability of $\Sigma$. This suggests that using 
realization theory of polynomial systems is likely to yield more useful results for 
ODE-LSTMs than using realization theory of general analytic systems. 

\section{Proofs}

\subsection{Proof of Theorem \ref{Theoremi:main}}\label{Proof1}

Let us prove that all ODE-RNN have polynomial embeddings.
Let $\Sigma = (A, B, C, \sigma, x_0)$ be an ODE-RNN with format $(m,n,\p)$.
Denote $L=Kn+n$ and consider the bijection
\begin{eqnarray*}
\phi:[K]\times [n] & \longrightarrow & [Kn] \\
(r,j) & \longmapsto & \phi(r,j):=r + K(j-1)
\end{eqnarray*}
For $L$ formal symbols $X_\gamma$,  let us write $X_{\phi(j,r)}$ for the unique 
index $\gamma \in [Kn]$ such that $\phi(j,r)=\gamma$.

Then one can construct an associated polynomial system 
$\mathscr{P}(\Sigma)=(\{P_\alpha\}_\alpha, h,  \upsilon_0)$ with
\begin{itemize}
\item $P_\alpha\in \R[X_1,\ldots,X_L;\R^{Kn}]$
\item $h\in \R[X_1,\ldots,X_L;\R^{\mathrm{p}}]$
\item $\upsilon_0 \in \R^L$.
\end{itemize}
as follows:
\begin{align*}
& P_{\phi(j,r),u(t)}= \tilde{P}(X_{j,r}) 
  \left(\sum_{k=1}^{n} a_{j,k} \tilde{P}_0(X_{k,u(t)})\right)   \\
 & P_{Kn+j,r}=\tilde{P}_0(X_{j,r}),  \\
 &  h_{k}=\sum_{j=1}^{n} c_{k,j} X_{j+Kn}\\
 & (\upsilon_0)_{\phi(j,r)}=\sigma \big( e_j^T(Ax_0 + B\alpha_r) \big), \\
 & (\upsilon_0)_{Kn+j}=e_j^Tx_0.
\end{align*}
where $k\in[\mathrm{p}], j\in[n ], r\in[K] $ and $u(t)$ ranges the $K$-
components of $\{P_\alpha\}_\alpha$ and where $\tilde{P},\tilde{P}_0$ are polynomials in one variables defining a polynomial system, which by \cite[Lemma 1]{TMArxive} is equivalent to Assumption \eqref{assum1}. 
By \cite[Lemma 2]{TMArxive},  if $(x,u,y)$ is a solution of a ODE-RNN $\Sigma$, 
then $(F(x),u,y)$ is a solution of $\mathscr{P}(\Sigma)$ where 
 $F:\mathbb{R}^n \rightarrow \mathbb{R}^L$ is given by  
 $F(x)=(z_1,\ldots,z_{nK},x^T)^T$, where 
 $z_{\phi(j,\alpha)}=\sigma(e_j^T(Ax+B\alpha))$.

\bigskip

Let us now prove that all ODE-LSTMs have polynomial embeddings.
Let $\Sigma = (\mathfrak{U}, \mathfrak{W}, \mathfrak{b}, C, \sigma, s_0)$ be a 
ODE-LSTM with format $(m,2n,\p)$, denote $U=\{\alpha_1,\ldots,\alpha_K\}$ its input 
space and $s(t)=(x(t)^T,z(t)^T)^T $ its state trajectory.  
Consider the ordering
\begin{eqnarray*}
\phi:[4]\times [n] \times [K] & \longrightarrow & [4nK], \\
(l,j,r) & \longmapsto & \phi(l,j,r):=j+n(l-1) + 4n(r-1).
\end{eqnarray*}
Recall the discussion after \eqref{LSTM_equation} the function $h:[0,+\infty 
[ \rightarrow \R^{n}$, 
i.e., $h(t)_j = z_{j}(t) \sigma_5(x_j(t))$, $j=1,2,\ldots,n$.  In particular, 
there exists a function
$\chi:\mathbb{R}^{2n} \rightarrow \mathbb{R}^n$, $\chi(s(t))=h(t)$. 
Clearly, we can defined a function $R^l:U \times \R^{2n} \rightarrow \R^n$ such that 
for $\alpha \in U$, $R(\alpha,s)$ is polynomial in $s$ and 
$$
R^l(\alpha,s(t))=\mathfrak{U}^l \underbrace{\chi(s(t))}_{h(t)}+ 
\mathfrak{W}^l \alpha +b^l,\quad l\in[4].
$$
By writing $\zeta_{ \phi(l,j,r)}(s(t))=\sigma_l(e^T_j(R^l(\alpha_r,s(t))))$ we can define
a map $F: \R^{2n} \rightarrow \R^{4nK+3n}$ such that
$$
F(s(t))=(\zeta_1(s(t)),\ldots,\zeta_{4nK}(s(t)),\overrightarrow{\sigma_5}(x(t)),s(t)).
$$
Define $ \upsilon(t) = F(s(t))$. 
Its coordinates are explicitly written as
\begin{equation} \label{map_proof_reach_obs}
\hspace*{-0.5cm}{\left\lbrace \begin{array}{ll}
\upsilon_{\phi(l,j,r)}(t) = \sigma_l \big( e_j^T \big( R^l(\alpha_r ,s(t))\big) 
\big) \, ,    \\[2mm]
\upsilon_{4nK+j}(t) = \sigma_5(x_j(t)) \, , \quad j\in[n] \\[2mm]
\upsilon_{4nK+n+k}(t) = s_k (t) \quad k\in[2n]
\end{array} \right.}
\end{equation}
Let us ease the notation and write 
\begin{itemize}
\item $ \upsilon_{ \phi(l,j,r)}(t)=\upsilon_{(l-1)n+j,r}(t)$,
\item $\upsilon_{\alpha_r}(t)=(\upsilon_{1,r}(t),\ldots,\upsilon_{4n, r}(t))$,
\item $\upsilon_{4nK+*}(t) = (\upsilon_{4nK+1}(t), \ldots, \upsilon_{4nK+n}(t))$,
\item $\upsilon(t) = (\upsilon_{\alpha_1}(t), \ldots, 
\upsilon_{\alpha_K}(t), \upsilon_{4nK+*}(t), x(t)^T, z(t)^T)^T$
\item $x_j(t) = \upsilon_{4nK+n+j}(t)$ and $z_j(t) = \upsilon_{4nK+2n+j}(t)$,
\end{itemize}
Set $\mathfrak{U}^l=(\mathfrak{U}^l_{i,j})_{i,j=1}^{n}$, $0 \leq l \leq 4$, and 
$\sigma_k^{(1)} = P_k(\sigma_k),k\in[5]$.
For each $r=1,2,\ldots, K$, define the polynomials $Q_{j,r}$, $Q_{5n+j,r}$, 
$Q_{ln+j,r}$, $j=1,\ldots,n$, $l \in [4]$, as follows:
\begin{itemize}
\item $Q_{j,r}(\upsilon(t))=\sum_{i=1}^n \mathfrak{U}^0_{j,i} x_i(t) + 
\upsilon_{n+j,r}(t) x_j(t) + \upsilon_{2n+j,
r}(t) \upsilon_{j,r}(t)$
\item $Q_{5n+j,r}(\upsilon(t))=Q_{j,r}(\upsilon(t)) P_5(\upsilon_{4nK+j}(t))$
\item for $l\in[4]$,  
\[ 
 \begin{split}
  & Q_{ln+j,r}(\upsilon(t))=P_{l}(\upsilon_{ \phi(l,j,r)}(t)) \times \\
& \sum_{i=1}^n\Big( \mathfrak{U}_{j,i}^{l} \big( \upsilon_{ \phi(4,i,r)}(t) 
\upsilon_{4nK+i}(t) + 
 z_j(t) Q_{i,r}(\upsilon(t)) P_5(\upsilon_{4nK+i}(t) \big) \Big) 
 \end{split}
\]
\end{itemize}
We set $ \upsilon(t)$ to be the state variables of the polynomial system 
$\mathscr{P}(\Sigma)=(\{Q_u\}_u, C,  \upsilon_0)$
defined,for $l\in[4],j \in[n]$, by
\begin{align*} \label{rational_system_R(Sigma)}
& \dot x_j(t) = Q_{j,u(t)}(\upsilon(t)) \, , ~ x_j(0)=e_j^Tx_0,\\
& \dot{z}_j(t) = \upsilon_{3n+j,u(t)}(t) \, , ~ z_j(0)=e_j^T z_0 \\[2mm]
& \dot \upsilon_{ (l-1)n+j,u(t)}(t) = Q_{ln+j,u(t)}(\upsilon(t))  \, , \\[2mm]
& \dot \upsilon_{4nK + j}(t) = Q_{5n+j,u(t)}(\upsilon(t)) \\[2mm]
& \upsilon_{  \phi(l,j,r)}(0) = \sigma_{l} \big( e_j^T(R^l(\alpha_r,s(0))) 
\big) \, , \hspace*{2mm} \\
& \upsilon_{4nK+j}(0) = \sigma_5(e_j^T x_0)\\
& h(0) = z_0 \odot \overrightarrow{\sigma_5}(x_0) \, , \\[2mm]
& y(t)=C s(t)
\end{align*}
with initial state $\upsilon_0 = \upsilon(0)$.

\begin{Lemma} \label{Theoremi:main:lemma}
Let $\upsilon(t) = F(s(t))$. If $(s,u,y)$ is a solution of an ODE-LSTM $\Sigma$, then $(\upsilon,u,y)$ is a solution of $\mathscr{P}(\Sigma)$.
\end{Lemma}

\begin{proof}
Let $(s,u,y)$ be a solution of $\Sigma$. To show that $(\upsilon,u,y)$ is a 
solution of $\mathscr{P}(\Sigma)$, it suffices to prove that $\upsilon$ 
satisfies the differential equation
\eqref{rational_system_R(Sigma)}.  For $u(t)\in U$,  we calculate the first time-derivative of 
$x_j(t)$. Following \eqref{LSTM_equation}, we know that we have
\begin{eqnarray*}
& & \dot{x}_j(t) \\
& =& \sum_{i=1}^n \mathfrak{U}^0_{j,i} x_i(t) + g^2_j(t) x_j(t) \\
& & + g^2_j(t) \sigma_1 \big( e_j^T(R^1(u(t) ,t)) \big)\\
& =& \sum_{i=1}^n \mathfrak{U}^0_{j,i} x_i(t) + \sigma_2 \big( e_j^T(R^2(u(t) ,t)) \big) x_j(t)\\
& & + \sigma_3 \big( e_j^T(R^3(u(t)  ,t)) \big) \, \sigma_1 \big( 
e_j^T(R^1(u(t)  ,t)) \big)\\
& = & Q_{j,u(t)}(\upsilon(t)) \, ,
\end{eqnarray*}
as desired.  Now notice that $\dot{z}_j(t) = 
\upsilon_{\phi(4,j,r)}(t)$, where $r,t$ are such that $u(t)=\alpha_r \in U$. Then we get
\begin{eqnarray*}
 \dot{\upsilon}_{4nK+j}(t) 
&=&  \frac{d}{dt} \big( \sigma_5(x_j(t)) \big) \\
&= & \dot{x}_j(t) P_5 \big( \sigma_5(x_j(t)) \big)\\
&=&  Q_{j,u(t)}(\upsilon(t)) P_5(\upsilon_{4nK+j}(t)) \, .
\end{eqnarray*}
Now take $l \in[4]$. We obtain
\begin{eqnarray*}
& & \dot{\upsilon}_{\phi(l,j,r)}(t)= \frac{d}{dt} \Big( \sigma_l 
\big( e_j^T(R^l(\alpha_r ,s(t))\big) \Big) \\
& =&  P_l(\upsilon_{\phi(l,j,r)}(t)) \Big( \sum_{i=1}^n 
\mathfrak{U}_{j,i}^l \dot{h}(t) \Big)\\
& =& P_l(\upsilon_{\phi(l,j,r)}(t)) \\
& &  \Big( \sum_{i=1}^n \mathfrak{U}_{j,i}^l \big(\dot{z}_i(t) \sigma_5(x_i(t)) +z_i(t) 
\dot{x_i}(t) P_5(\sigma_5(x_i(t)))\big) \Big) \\
& =& P_l(\upsilon_{(\phi(l,j,r)}(t))\Big( \sum_{i=1}^n 
\mathfrak{U}_{j,i}^l \big( \upsilon_{3n+i,u(t)}(t) \upsilon_{4nK+i}(t) \\
& &  \qquad \qquad \qquad \quad + z_i(t) Q_{i,u(t)}(\upsilon(t)) 
P_5(\upsilon_{4nK+i}(t)) \big) \Big)   \\
& =& Q_{ln+j,u(t)}(\upsilon(t)) \, .
\end{eqnarray*}
This completes the proof of Lemma \ref{Theoremi:main:lemma}.
\end{proof}

We observe that the map $F$ is a 
smooth map, in particular it is a continuous map and Lemma 
\ref{Theoremi:main:lemma} showed that, for all $u \in \mathcal{U}_{pco}$, if $
(s,u,y)$ is a solution of  $\Sigma$, then $(\upsilon,u,y)$ is a solution of 
$\mathscr{P}(\Sigma)$, with 
\begin{equation} \label{equation_proof_reach_obs}
 \upsilon(t) = F(s(t)) \, ,  \quad \forall t \geqslant 0,
\end{equation}
where $s(t) = (x(t)^T,z(t)^T)^T \in \R^{2n}$.  

\medskip

Finally,  if an i-o map $p: \mathcal{U}_{pco} \rightarrow \mathcal{S}_{\on{pc}}$ is 
realized by a system $\Sigma$ in $\F_1$ or $\F_2$, then the polynomial embedding 
$\mathscr{P}(\Sigma)$ also realizes $p$ by \cite[Lemma 2]{TMArxive} and Lemma 
\ref{Theoremi:main:lemma}. Thus $p$ is causal and analytic and $\on{trdeg}
\mathcal{A}_{obs}(p)<\infty$ (by \cite[Theorem 3]{Bartoszewicz}, 
\cite[Theorem 5.16]{JanaSIAM}).

This concludes the proof of Theorem \ref{Theoremi:main}.

\medskip


\subsection{Proof of Proposition \ref{RNN_ReachObs}}\label{Proof2}

First,  we prove that, if $\mathscr{P}(\Sigma)$ is accessible, then
$\Sigma$ is also accessible.  By definition of the map $F$ from 
\eqref{map_proof_reach_obs}, we obtain
\begin{equation} \label{map_onto_reach_set}
F \big( R_{\Sigma}(s_0) \big) = R_{\mathscr{P}(\Sigma)}(\upsilon_0) \, ,
\end{equation}
where $R_{\Sigma}(s_0)$ and $R_{\mathscr{P}(\Sigma)}(\upsilon_0)$ are 
respectively the reachable set of $\Sigma$ and of $\mathscr{P}(\Sigma)$, and $
\upsilon_0 = \upsilon(0)$. Now suppose that the polynomial system $\mathscr{P}
(\Sigma)$ is accessible, i.e. there exists a non-empty open set $O$ included in 
$R_{\mathscr{P}(\Sigma)}(\upsilon_0)$. Thus $F^{-1}(O)$ is a non-empty open 
set (because $F$ is a continuous map) included in $R_{\Sigma}(s_0)$, so that 
$\Sigma$ is accessible.

\medskip

Next,  we prove that, if $\mathscr{P}(\Sigma)$ is algebraically reachable, then 
$\Sigma$ is span-reachable. Suppose that $\mathscr{P}(\Sigma)$ is algebraically 
reachable, i.e. there is no non-trivial polynomial which vanishes on the reachable set 
$R_{\mathscr{P}(\Sigma)}(\upsilon_0)$. Take $u \in \mathcal{U}_{pco}$ such that 
$(s,u,y)$ a solution of $\Sigma$, with 
$s(t)=(x(t)^T,z(t)^T)^T \in \R^{2n}$, for $t \geqslant 0$. Consider $(\upsilon,u,y)$ 
the solution of $\mathscr{P}(\Sigma)$ obtained by Lemma 
\ref{Theoremi:main:lemma}. Assume that $\Sigma$ is not span-reachable, i.e. there exist 
reals $\lambda_1, \ldots, \lambda_{2n}$ such that 
\begin{equation*}
\displaystyle \sum_{j=1}^n \lambda_j x_j(t) + \sum_{j=1}^n \lambda_{n+j} z_j(t) 
= 0 \, .
\end{equation*}
Then taking the first derivative of the above equation gives
\begin{equation*}
\displaystyle \sum_{j=1}^n \lambda_j Q_{j,u(t)}(\upsilon(t)) + \sum_{j=1}^n 
\lambda_{n+j} \upsilon_{3n+j,u(t)}(t) = 0 \, ,
\end{equation*}
which is a contradiction, because there exists at least one non-trivial polynomial 
(given by the above equation) vanishing on the reachable set 
$R_{\mathscr{P}(\Sigma)}(\upsilon_0)$.

\medskip

Next,  we prove that, if $\mathscr{P}(\Sigma)$ is observable, the
$\Sigma$ is also observable. Take $s_0, s_0' \in \R^{2n}$ two initial states of 
$\Sigma$ such that $s_0 \neq s_0'$. Thus we have 
$\upsilon_0 = F(s_0) \neq F(s_0') = \upsilon_0'$, 
because the map $F$, defined in \eqref{map_proof_reach_obs}, is injective. As 
$\mathscr{P}(\Sigma)$ is observable, there exist solutions $(\upsilon, u, y)$ and 
$(\upsilon', u, y')$ 
of $\mathscr{P}(\Sigma)$ such that $\upsilon(0)=\upsilon_0$ and 
$\upsilon'(0) = \upsilon_0'$, 
and $y \neq y'$. By Lemma \ref{Theoremi:main:lemma} and by 
\eqref{map_onto_reach_set}, there exist solutions $(s,u,y)$ and $(s',u,y')$ of 
$\Sigma$ 
with $s(0)=s_0$ and $s'(0)=s_0'$, satisfying 
\begin{equation} \label{equation_obs_solutions_proof_reach_obs}
\forall t \geqslant 0 \, , \quad F(s(t)) = \upsilon(t) \, , \quad \mbox{and} \quad F(s'(t)) 
= \upsilon'(t) \, , 
\end{equation}
and $y \neq y'$. Thus $\Sigma$ is observable, as desired. Now if 
$\mathscr{P}(\Sigma)$ is algebraically reachable, then it is observable, by 
\cite[Proposition 3]{Bartoszewicz}. Then $\Sigma$ is observable, 
by the above arguments.

\medskip

Finally,   we prove that, if $\mathscr{P}(\Sigma)$ is semi-algebraically observable, 
then $
\Sigma$ is weakly observable. By \cite[Proposition 4.20, Corollary 4.22]{NashSIAM}, 
the polynomial system $\mathscr{P}(\Sigma)$ is weakly observable. Let 
$s_0 \in \R^{2n}$ be an initial state of $\Sigma$, and let $\upsilon_0 = F(s_0)$. As 
$\mathscr{P}(\Sigma)$ is weakly observable, there exist a open set $V$ with 
$\upsilon_0 \in V$ such 
that, for all $\upsilon_0' \neq \upsilon_0 \in V$, there exist solutions $(\upsilon,u,y)
$ and $(\upsilon',u,y')$ satisfying $\upsilon(0)=\upsilon_0$, 
$\upsilon'(0)=\upsilon_0'$, and $y \neq y'$. We set $U = F^{-1}(V)$ which is an 
open set (because $F$ is a continuous map) 
such that $s_0 \in U$. Take $s_0' \neq s_0 \in U$, and set 
$\upsilon_0' = F(s_0') \in V$. 
By injectivity of $F$, $\upsilon_0' \neq \upsilon_0 \in V$. Thus we can find 
solutions $(\upsilon,u,y)$ and $(\upsilon',u,y')$ of $\mathscr{P}(\Sigma)$ as 
above. Then there exist 
solutions $(s,u,y)$ and $(s',u,y')$ of $\Sigma$ such that 
\eqref{equation_obs_solutions_proof_reach_obs} holds. We know that $y \neq y'$. 
Thus $\Sigma$ is weakly observable. 

\section{Examples}

Linear systems are a particular case of ODE-LSTMs, by taking $\sigma_1$ the identity map, 
$\sigma_2 = \sigma_4 = 0$ (the constant functions equal to $0$), $\sigma_3 = 1$ (the 
constant function equal to $1$), $b^1\in \mathbb{R}^n$ the trivial vector and $C \in 
\mathbb{R}^{p \times 2n}$ a matrix of the form $(\widetilde{C}, 0)$ with $\widetilde{C} \in 
\mathbb{R}^{p \times n}$. Also ODE-RNNs are particular cases of ODE-LSTMs by taking $\sigma_2,=
\sigma_4=0$, $\sigma_3=1$, $\sigma_1$ to be any non-constant continuous globally 
Lipschitz function, and taking $b_1 $ to be trivial.

\begin{Remark}\label{eqn:resnet}
For a suitable choice of $C$ and $U^0$,  and taking $\sigma_3 = \sigma_4 = \sigma_2$, the Euler discretization of an ODE-LSTM in $\F_2$ is given by
\begin{equation*}
\Sigma_{discretized}: \left\lbrace \begin{array}{lllll} 
x(k+1) = x(k) + f(k) \odot x(k) + i(k) \odot g^1(k) \\[2mm]
f(k) = \overrightarrow{\sigma_2} \big( \mathfrak{U}^2 h(k) + \mathfrak{W}^2 u(k) + b^2 \big)\\[2mm]
i(k) = \overrightarrow{\sigma_3} \big( \mathfrak{U}^3 h(k) + \mathfrak{W}^3 u(k) + b^3 \big)\\[2mm]
z(k+1) = z(k)+ \overrightarrow{\sigma_2} \big( \mathfrak{U}^4 h(k) + \mathfrak{W}^4 u(k) + b^4)\\[2mm]
h(k) = z(k) \odot \overrightarrow{\sigma_5}(x(k))\\[2mm]
x(0) = x_0 \, , \; z_0 = z(0)  \\[2mm]
y(k) = z(k). \end{array} \right.
\end{equation*}
This is closely related to LSTM networks defined in \cite{Gers}, where, at the 
$k$th step, $x(k), f(k), i(k)$ are usually 
called respectively the cell, the forget gate and the input gate
and $u(k), z(k)$ respectively the input and the output.  The presence of a skip 
connection makes this actually a \textit{residual} LSTM which, being the 
discretization of an ODE, enjoys of (gradient) stable dynamics, contrary to 
vanilla LSTMs. In addition, our construction should be readily applicable to 
LEM networks defined in \cite{rusch} which are also presented as discretized 
two-gated recurrent neural ODEs and to State-Space models \cite{gu2021combining, gu2022efficiently,goel2022sashimi}.
\end{Remark}

\medskip

\begin{Example}
Let us exhibit an ODE-LSTM whose 
polynomial embedding is minimal. Set $\mathcal{U} = \{u\}\subset \R$ and consider:
\begin{equation} \label{LSTM_example_1}
\Sigma ~ : ~ \left\lbrace \begin{array}{lllll}
\dot{x}(t) = \sigma(x(t) z(t) + u) x(t)\\
\dot{z}(t) = 0\\
x(0)=0 \, , z(0)=a \, , ~ \mbox{with} ~ a \neq 0\\
y(t) = x(t),
\end{array} \right. 
\end{equation}
where $\sigma$ is the sigmoid function. Here $\sigma_1 = \sigma_3 = \sigma_4 = 0$, 
$\sigma_5$ is the identity, $U^2 = 1$, $W^2 = 1$, $b^2$ is the zero vector and 
$C = (1, 0) \in \mathbb{R}^{1 \times 2}$. 
We can rewrite \eqref{LSTM_example_1} as:
\begin{equation}
\Sigma ~ : ~ \left\lbrace \begin{array}{lllll}
\dot{x}(t) = \sigma(a x(t) + u) x(t),\\
x(0)=0,\\
y(t) = x(t).
\end{array} \right. \, 
\end{equation}
Then, as $m(\sigma) = 2$, $\mathscr{P}(\Sigma)$ is given by
\begin{equation*}
\mathscr{P}(\Sigma) ~ : ~ \left\lbrace \begin{array}{lllll}
\dot{\upsilon}_{1,u} = a \upsilon_2 \upsilon_{1,u} (1 - \upsilon_{1,u}),\\
\dot{\upsilon_2} = \upsilon_{1,u} \upsilon_2,\\
\upsilon_{1,u}(0)=\sigma(u) \, , \, \upsilon_2(0)=0 \, ,\\
y(t) =h(\upsilon(t)) = \upsilon_2(t).
\end{array} \right. \, 
\end{equation*}
where $\upsilon(t) = (\upsilon_{1,u}(t),\upsilon_2(t))^T$ 
$\upsilon_{1,u}(t) = \sigma(ax(t) + u)$ and $\upsilon_2(t) = x(t)$, for $t \geqslant 0$.  
It is clear that $\upsilon_{1,u}$ and 
$\upsilon_2$ belong to the observation field $\mathcal{Q}_{obs}(\mathscr{P}(\Sigma))$.

Now we prove that $dim(\mathscr{P}(\Sigma)) =2$. Observe that
\begin{equation*}
\forall t \geqslant 0 \, , \quad \frac{\dot{\upsilon}_{1,u}(t)}{1 - \upsilon_{1,u}(t)} 
= a \dot{\upsilon_2}(t), 
\end{equation*}
and notice that $\upsilon_{1,u}(t) < 1$, because the sigmoid function takes value in $]0;1[$. 
By taking the primitives of both sides of the above equation, there is $c \in \mathbb{R}$ 
such that
\begin{equation*}
\forall t \geqslant 0 \, , \quad ln(1 - \upsilon_{1,u}(t)) = a \upsilon_2(t) + c \, ,
\end{equation*}
which implies that
\begin{equation*}
\forall t \geqslant 0 \, , \quad \upsilon_{1,u}(t) = 1 - K e^{a \upsilon_2(t)} \, , ~ \mbox{with} ~ K 
= e^c \, .
\end{equation*}
Thus, there is no non-trivial polynomial which vanishes on the set of reachable states of 
$\mathscr{P}(\Sigma)$, i.e. $\mathscr{P}(\Sigma)$ is algebraically reachable.  We conclude 
that $dim(\mathscr{P}(\Sigma)) = 2 = m(\sigma)$ which implies that $\Sigma$ is minimal.
\end{Example}

\medskip 

In what follows we set $\mathcal{U} = \{\alpha_1, \alpha_2\} \subset \mathbb{R}$ with 
$\alpha_1 \neq \alpha_2$.

\begin{Example}
Let us exhibit a reduction method of polynomial ODE-LSTM embeddings. Consider:
\begin{equation*} \label{LSTM_example_variables_rewritten}
\Sigma ~ : ~ \left\lbrace \begin{array}{lllll}
\dot{x}(t) = \sigma(h(t) + u(t)) x(t) + \sigma(h(t) + u(t)),\\
\dot{z}(t) = \sigma(h(t) + u(t)),\\
h(t) = x(t) z(t),\\
x(0)=0 \, , z(0)=0,\\
y(t) = z(t),
\end{array} \right. \, 
\end{equation*}
where we take $\sigma_1 = \sigma_2 = \sigma_4 = \sigma$ to be the sigmoid function, 
$\sigma_3=1$ and $\sigma_5$ the identity map. Here $n(\sigma)=2$, $k(\sigma) = 1$ so that 
$m(\sigma)=4$.  For $k=1,2$ and $t \geqslant 0$,  using Lemma \ref{Theoremi:main:lemma},  let us set
$\upsilon_{1,\alpha_k}(t) = \sigma_1(x(t) z(t) + \alpha_k) = \sigma(x(t) z(t) + \alpha_k) 
= \upsilon_{2,\alpha_k}(t) = \upsilon_{4, \alpha_k}(t)$, $\upsilon_{3,\alpha_k}(t) = 1$, and 
$\upsilon_5(t) = x(t)$.  Thus,  $\mathscr{P}(\Sigma)$ is simply given by
\begin{equation*}
 \left\lbrace \begin{array}{lllll}
\dot{\upsilon}_{1,\alpha_1} = (\upsilon_{1,u(t)} \upsilon_2 \upsilon_3 + 
\upsilon_{1,u(t)} \upsilon_3 + 
\upsilon_{1,u(t)} \upsilon_2) \upsilon_{1,\alpha_1} (1 - \upsilon_{1,\alpha_1}) \, ,\\
\dot{\upsilon}_{1,\alpha_2} = (\upsilon_{1,u(t)} \upsilon_2 \upsilon_3 + 
\upsilon_{1,u(t)} \upsilon_3 + 
\upsilon_{1,u(t)} \upsilon_2) \upsilon_{1,\alpha_2} (1 - \upsilon_{1,\alpha_2}) \, ,\\
\dot{\upsilon}_2 = \upsilon_{1,u(t)} \upsilon_2 + \upsilon_{1,u(t)} \, ,\\
\dot{\upsilon}_3 = \upsilon_{1,u(t)} \, , \\
\upsilon_{1,\alpha_k}(0)=\sigma(\alpha_k) , \, \upsilon_2(0) =
 \upsilon_3(0) = 0 \, , ~ \mbox{for} ~ k=1,2 \, ,\\
y(t) = \upsilon_3(t),
\end{array} \right. \, 
\end{equation*}
where we set $x(t)=\upsilon_2(t)$ and $z(t) = \upsilon_3(t)$.
\end{Example}

\begin{Example}
Let us now exhibit an accessible ODE-LSTM whose polynomial embedding is not accessible:
\begin{equation} \label{LSTM_counterexample_1}
\Sigma ~ : ~ \left\lbrace \begin{array}{lllll}
\dot{x}(t) = \sigma(h(t) + u(t))^2 \, ,\\
\dot{z}(t) = \sigma(h(t) + u(t)) \, ,\\
h(t) = x(t) z(t) \, , \\
s_0=(x(0),z(0))=(0,0)^T \, ,\\
y(t) = z(t).
\end{array} \right. \, 
\end{equation}
In this case, $n(\sigma)=2$. We first prove that $\Sigma$ is accessible.  Let 
$f_{\alpha_1}, f_{\alpha_2}: \mathbb{R}^2 \rightarrow \mathbb{R}^2$ be vector fields 
generated by $\Sigma$. We denote by $\mathcal{L}_{\Sigma}(s_0)$ the smallest Lie algebra 
containing $f_{\alpha_1}, f_{\alpha_2}$ and closed by Lie brackets.  We then have
\begin{equation*} \label{equation_LSTM_counterexample1}
\left\lbrace \begin{array}{ll}
\frac{1}{\sigma(\alpha_2)} f_{\alpha_2}(s_0) - \frac{1}{\sigma(\alpha_1)} f_{\alpha_1}(s_0) 
= (\sigma(\alpha_2)-\sigma(\alpha_1)) \; (1, 0)^T \in \mathbb{R}^2 \, , \\[2mm]
\frac{1}{\sigma(\alpha_2)^2} f_{\alpha_2}(s_0) - \frac{1}{\sigma(\alpha_1)^2} f_{\alpha_1}
(s_0) = \frac{\sigma(\alpha_1)-\sigma(\alpha_2)}{\sigma(\alpha_1) 
\sigma(\alpha_2)} \; (0, 1)^T \in \mathbb{R}^2.
\end{array} \right. 
\end{equation*}
As $\sigma(\alpha_1), \sigma(\alpha_2) > 0$ and $\sigma(\alpha_1) \neq \sigma(\alpha_2)$ 
because $\alpha_1 \neq \alpha_2$ and $\sigma$ is bijective and takes values in $]0;1[$, 
then $dim \mathcal{L}_{\Sigma}(s_0) = 2 = n(\sigma)$. 
By \cite[Theorem 3.10]{J-controllability}, $\Sigma$ is accessible. 
Now $\mathscr{P}(\Sigma)$ is given by
\begin{equation*}
 \left\lbrace \begin{array}{llllll}
\dot{\upsilon}_{1,\alpha_1} = (\upsilon_{1,u(t)}^2 \upsilon_3 + 
\upsilon_{1,u(t)} \upsilon_2) \upsilon_{1,\alpha_1} (1 - \upsilon_{1,\alpha_1}) \, ,\\
\dot{\upsilon}_{1,\alpha_2} = (\upsilon_{1,u(t)}^2 \upsilon_3 + 
\upsilon_{1,u(t)} \upsilon_2) \upsilon_{1,\alpha_2} (1 - \upsilon_{1,\alpha_2}) \, ,\\
\dot{\upsilon_2} = \upsilon_{1,u(t)}^2 \, ,\\
\dot{\upsilon_3} = \upsilon_{1,u(t)} \, ,\\ 
\upsilon_{1,\alpha_k}(0)=\sigma(\alpha_k) , \, \upsilon_2(0) =
 \upsilon_3(0) = 0 \, , ~ \mbox{for} ~ k=1,2 \, ,\\
y(t) = \upsilon_3(t),
\end{array} \right. \, 
\end{equation*}
where, for $t \geqslant 0,\alpha \in \mathcal{U}$,  we set  
$\upsilon_{1,\alpha}(t) = \sigma(x(t)z(t) + \alpha)$, $\upsilon_2(t) = x(t)$ and 
$\upsilon_3(t) = z(t)$. We denote
$\upsilon_0 = (\sigma(\alpha_1),\sigma(\alpha_2),0,0)^T \in \mathbb{R}^4$ the initial 
state of $\mathscr{P}(\Sigma)$ and let
$g_{\alpha_1}, g_{\alpha_2}: \mathbb{R}^4 \rightarrow \mathbb{R}^4$ be vector fields 
generated by the polynomial system $\mathscr{P}(\Sigma)$.  We denote
$\mathcal{L}_{\mathscr{P}(\Sigma)}(\upsilon_0)$ the smallest Lie algebra 
containing $g_{\alpha_1}, g_{\alpha_2}$ and closed by Lie brackets. It is easy to prove that 
$(g_{\alpha_1}(\upsilon_0), g_{\alpha_2}(\upsilon_0))$ is linearly independent, so that 
$2 \leqslant dim \mathcal{L}_{\mathscr{P}(\Sigma)}(\upsilon_0)$.
\end{Example}

\begin{Example}
Let us exhibit an ODE-LSTM $\Sigma$ which, seen as an analytic system, is both accessible and 
weakly observable and thus is minimal:
\begin{equation} \label{LSTM_example_2}
\Sigma ~ : ~ \left\lbrace \begin{array}{lllll}
\dot{x}(t) = \sigma(h(t) + u(t)) x(t) + \sigma(h(t) + u(t))^2,\\
\dot{z}(t) = \sigma(h(t) + u(t)),\\
h(t) = x(t) z(t),\\
x(0)=0 \, , z(0)=0,\\
y(t) = z(t),
\end{array} \right. \, 
\end{equation}
where $\sigma_1 = \sigma_2 = \sigma_3 = \sigma_4 =\sigma$ is the sigmoid function and 
$\sigma_5$ the identity map. Here $n(\sigma) = 2$ and $\mathscr{P}(\Sigma)$ is given 
by:
\begin{equation*}
\left\lbrace \begin{array}{llllll}
\dot{\upsilon}_{1,\alpha_1} = (\upsilon_{1,u(t)} \upsilon_2 \upsilon_{3} + 
\upsilon_{1,u(t)}^2 \upsilon_3 + 
\upsilon_{1,u(t)} \upsilon_2) \upsilon_{1,\alpha_1} (1 - \upsilon_{1,\alpha_1}) \, ,\\
\dot{\upsilon}_{1,\alpha_2} = (\upsilon_{1,u(t)} \upsilon_2 \upsilon_{3} +
 \upsilon_{1,u(t)}^2 \upsilon_3 + 
 \upsilon_{1,u(t)} \upsilon_2) \upsilon_{1,\alpha_2} (1 - \upsilon_{1,\alpha_2}) \, ,\\
\dot{\upsilon_2} = \upsilon_{1,u(t)} \upsilon_2 + \upsilon_{1,u(t)}^2 \, ,\\
\dot{\upsilon_3} = \upsilon_{1,u(t)} \, ,\\ 
\upsilon_{1,\alpha_1}(0) = \sigma(\alpha_1) \, , \upsilon_{1,\alpha_2}(0) 
= \sigma(\alpha_2) \, , \upsilon_2(0)=\upsilon_3(0)=0 \, ,\\
y(t) = \upsilon_3(t).
\end{array} \right. \, 
\end{equation*}
Now, for $t \geqslant 0$,  set 
$\upsilon_{1,\alpha_k}(t) = \sigma(x(t) z(t) + \alpha_k)$, where $k=1,2$, and set 
$\upsilon_2(t) = x(t)$ and $\upsilon_3(t) = z(t)$.\\
Notice that equations \eqref{LSTM_counterexample_1} hold in this case, so 
$dim \mathcal{L}_{\Sigma}(s_0) = 2 = n(\sigma)$.  Then, by 
\cite[Theorem 3.10]{J-controllability},  we conclude that $\Sigma$ is accessible (i.e. its 
reachable set from $s_0 = (0,0)^T$ contains a non-empty open set).

Now let's prove that $\Sigma$ is weakly observable by proving that 
$\mathscr{P}(\Sigma)$ is semi-algebraically observable.  Denote 
$g_{\alpha_1}, g_{\alpha_2}: \mathbb{R}^4 \rightarrow \mathbb{R}^4$ 
the vector fields generated by $\mathscr{P}(\Sigma)$, and 
$L_{g_{\alpha}}$ the Lie derivative operator along $g_{\alpha}$, 
for $\alpha \in \mathcal{U}$. The output map of 
$\mathscr{P}(\Sigma)$ is $h = \upsilon_3$.  It is then clear that $\upsilon_3$ belongs to 
the observation algebra of $\mathscr{P}(\Sigma)$. Moreover we have
\begin{align*}
L_{g_{\alpha_1}} h = \upsilon_{1,\alpha_1} \, , \quad L_{g_{\alpha_2}} h 
= \upsilon_{1,\alpha_2} \, , 
\end{align*}
which shows that $\upsilon_{1,\alpha_1}, \upsilon_{1,\alpha_2}$ also belong to the 
observation algebra $\mathcal{A}_{obs}(\mathscr{P}(\Sigma))$ of $\mathscr{P}(\Sigma)$ 
as the latter algebra is closed under Lie 
derivatives along $g_{\alpha_1}, g_{\alpha_2}$. Thus, in 
$\mathcal{Q}_{obs}(\mathscr{P}(\Sigma))$ we have
\begin{equation*}
\upsilon_2 = \frac{\frac{L_{g_{\alpha_2}} L_{g_{\alpha_1}} h}{L_{g_{\alpha_1} 
h (1 - L_{g_{\alpha_1}})}} - h (L_{g_{\alpha_2}} h)^2}{h (L_{g_{\alpha_2}} h) 
+ L_{g_{\alpha_2}} h}\, .
\end{equation*}
This proves that $\mathscr{P}(\Sigma)$ is semi-algebraically observable. Thus $\Sigma$ is 
weakly observable by Lemma \ref{RNN_ReachObs}. As $\Sigma$ is seen as an 
analytic system and is accessible and weakly observable,  it is then minimal.
\end{Example}

\section{Conclusions and perspectives}

We have shown that i-o maps realized by large classes of recurrent neural ODEs 
(namely  ODE-RNNs and ODE-LSTMs) can be represented by polynomial systems, and we 
used this fact to derive necessary and sufficient conditions for the
existence of realizations by such systems and their minimality.  
Future research will be directed towards improving these results to derive 
a complete realization theory for ODE-LSTMs and apply them to formulating theoretical
guarantees for learning ODE-LSTMs. 

\section*{Acknowledgments} The first author thanks Pierre Marion for useful comments which helped improving the clarity of the exposition. This work has been supported by the French government under the "France 2030” program, as part of the SystemX Technological Research Institute.

\end{document}